\documentclass[12pt,a4paper]{article}
\usepackage{a4wide,amsmath,amsthm,amssymb,amsfonts,color}
\usepackage{mathtools}
\usepackage{hyperref} 
\mathtoolsset{showonlyrefs}
\usepackage[all]{xy}
\usepackage[utf8]{inputenc}
\usepackage{enumitem,MnSymbol,authblk}

\usepackage{tikz,tikz-3dplot}
\usepackage{mathrsfs}
\usepackage{cancel}

\newtheorem{thm}{Theorem}[section]
\newtheorem{prop}[thm]{Proposition}

\newtheorem{cor}[thm]{Corollary}
\newtheorem{dfn}[thm]{Definition}

\theoremstyle{remark}
\newtheorem{rmkk}[thm]{Remark}
\newtheorem{exe}[thm]{Example}
\newenvironment{rmk}{\begin{rmkk}\rm}{\qee\end{rmkk}}
\newenvironment{ex}{\begin{exe}\rm}{\qee\end{exe}}

\newcommand{\qee}{\mbox{\hspace{0.2mm}}\hfill$\triangle$}
\newcommand{\Z}{\mathbb{Z}}
\newcommand{\R}{\mathbb{R}}
\newcommand{\C}{\mathbb{C}}
\newcommand{\Q}{\mathbb{Q}}
\newcommand{\Pj}{\mathbb{P}}

\newcommand{\N}{\mathbb{N}}
\newcommand{\cO}{\mathcal{O}}

\title{\bf \large A WEAK $(k,k)$-LEFSCHETZ THEOREM FOR PROJECTIVE\\ TORIC ORBIFOLDS} 

\author[]{William D. Montoya }

\affil[]{Instituto de Matemática, Estatística e Computação Científica, \par\vskip-3pt Universidade Estadual de Campinas (UNICAMP),\par\vskip-3pt Rua Sérgio Buarque de Holanda 651,
13083-859, Campinas, SP, Brazil}

\begin{document}
\maketitle

\begin{abstract}Firstly we show a generalization of the $(1,1)$-Lefschetz theorem for projective toric orbifolds and secondly  we prove that on $2k$-dimensional quasi-smooth hypersurfaces coming from quasi-smooth intersection surfaces, under the Cayley trick, every rational $(k,k)$-cohomology class is algebraic, i.e., the Hodge conjecture holds on them.
\end{abstract}

\let\svthefootnote\thefootnote
\let\thefootnote\relax\footnote{
\hskip-2\parindent 
% SISSA Preprint 08/2019/MATE \\ 
Date:  \today  \\
{\em 2020 Mathematics Subject Classification:}  14C30, 	14M10, 14J70, 14M25 \\ 
{\em Keywords:} (1,1)- Lefschetz theorem, Hodge conjecture, toric varieties, complete intersection \\
Email: {\tt wmontoya@ime.unicamp.br}
}
\addtocounter{footnote}{-1}\let\thefootnote\svthefootnote

\section{Introduction}
In \cite{BruzzoMontoya} we proved that, under suitable conditions, on a very general codimension $s$ quasi-smooth intersection subvariety $X$ in a projective toric  orbifold $\Pj^d_{\Sigma}$ with $d+s=2(k+1)$ the Hodge conjecture holds, that is, every $(p,p)$-cohomology class, under the Poincaré duality is a rational linear combination of fundamental classes of algebraic subvarieties of $X$. The proof of the above-mentioned result relies, for $p\neq d+1-s$, on a Lefschetz theorem (\cite{Mavlyutov}) and the Hard Lefschetz theorem for projective orbifolds (\cite{Wang_2009}). When $p=d+1-s$ the proof relies on the Cayley trick, a trick  which associates to $X$ a quasi-smooth hypersurface $Y$ in a projective vector bundle, and the Cayley Proposition  (\ref{Cayley})   which gives an isomorphism of some primitive cohomologies (\ref{primitive}) of $X$ and $Y$. The Cayley trick, following the philosophy of Mavlyutov in \cite{Mavlyutov}, reduces results known for quasi-smooth hypersurfaces to quasi-smooth intersection subvarieties. The idea in this paper goes  the other way around, we translate some results for quasi-smooth intersection subvarieties to quasi-smooth hypersurfaces, mainly the $(1,1)$-Lefschetz theorem.

\textbf{Acknowledgement.} I thank Prof. Ugo Bruzzo and Tiago Fonseca for useful discussions. I also acknowledge support from FAPESP postdoctoral grant No. 2019/23499-7.

\section{Preliminaries and Notation} 

\subsection{Toric varieties}
Let $M$ be a free abelian group of rank $d$, let
$N = Hom(M, Z)$, and $N_{\R} = N \otimes_{\Z} \R.$

\begin{dfn}
 \begin{itemize}
     \item A convex subset $\sigma \subset N_R$ is a rational $k$-dimensional simplicial cone if there exist $k$ linearly independent primitive elements $e_1,\dots, e_k \in N$ such that $\sigma = \{\mu_1e_1 +\cdots+\mu_ke_k\}$. 
     
     \item The generators $e_i$ are  integral if for every $i$ and any nonnegative rational number $\mu$ the product $\mu e_i$ is in $N$ only if $\mu$ is an integer.
     
     \item Given two rational simplicial cones $\sigma$, $\sigma'$ one says that $\sigma'$ is a face of $\sigma$ ($\sigma'< \sigma$)
     if the set of integral generators of $\sigma'$ is a subset of the set of integral generators of $\sigma$.
     
     \item  A finite set $\Sigma = \{\sigma_1,\dots, \sigma_t\}$ of rational simplicial cones is called a rational simplicial complete $d$-dimensional fan if:
\begin{enumerate}
    \item all faces of cones in $\Sigma$ are in $\Sigma$;
    \item if $\sigma,\sigma'\in \Sigma$ then $\sigma\cap \sigma'< \sigma$ and $\sigma\cap \sigma' <\sigma'$;

 \item $N_{\R}=\sigma_1\cup \dots \cup \sigma_t$.
\end{enumerate}     
\end{itemize}
\end{dfn}

A rational simplicial complete $d$-dimensional fan $\Sigma$ defines a $d$-dimensional toric variety $\Pj_{\Sigma}^d$ having only orbifold singularities which we assume to be projective. Moreover, $T:=N\otimes_{\Z} \C^{*} \simeq (\C^{*})^d$ is the torus action on $\Pj_{\Sigma}^d$.
We denote by $\Sigma(i)$ the $i$-dimensional cones of $\Sigma$ and each $\rho\in \Sigma$ corresponds to an irreducible $T$-invariant Weil divisor $D_{\rho}$ on $\Pj_{\Sigma}^d$. Let ${\rm Cl}(\Sigma)$ be the group of Weil divisors on $\Pj_{\Sigma}^d$ module rational equivalences.

The total coordinate ring of $\Pj_{\Sigma}^d$ is the polynomial ring $S=\C[x_{\rho}\mid \rho\in\Sigma(1)]$,  $S$ has the ${\rm Cl}(\Sigma)$-grading, a Weil divisor $D=\sum_{\rho\in\Sigma(1)}u_{\rho} D_{\rho}$ determines the monomial $x^u:=\prod_{\rho\in\Sigma(1)}x_{\rho}^{u_{\rho}}\in S   $ and conversely $\deg(x^u)=[D]\in {\rm Cl}(\Sigma)$.

For a cone $\sigma\in \Sigma$, $\hat{\sigma}$ is the set of $1$-dimensional cone in $\Sigma$ that are not contained in $\sigma$ and $x^{\hat{\sigma}}:=\prod_{\rho\in \hat{\sigma}}x_{\rho}$ is the associated monomial in $S$. 

\begin{dfn} The irrelevant ideal of $\Pj_{\Sigma}^d$ is the monomial ideal $B_{\Sigma}:=<x^{\hat{\sigma}}\mid \sigma \in \Sigma>$ and the zero locus $Z(\Sigma): =\mathbb{V}(B_{\Sigma})$ in the affine space $\mathbb{A}^d:=Spec(S)$ is the irrelevant locus.
  
\end{dfn}

\begin{prop}[Theorem 5.1.11 \cite{CoxLittleSchenck}] The toric variety $\Pj_{\Sigma}^d$ is a categorical quotient $\mathbb{A}^d\setminus Z(\Sigma) $ by the group $Hom({\rm Cl}(\Sigma), \C^{*})$ and the group action is induced by the ${\rm Cl}(\Sigma)$-grading of $S$.

\end{prop}

\subsection{Orbifolds}
Now we give a brief introduction to complex orbifolds and we mention the needed theorems for the next section. Namely: de Rham theorem and  Dolbeault theorem for complex orbifolds. 

\begin{dfn} 
A complex orbifold of complex dimension $d$ is a singular complex space whose singularities are locally isomorphic to quotient singularities $\C^d/ G$, for finite subgroups $G\subset Gl(d,\C)$.
\end{dfn}

\begin{dfn} A differential form on a complex orbifold $Z$ is defined locally at $z\in Z$ as a $G$-invariant differential form on $\C^d$ where $G\subset Gl(d,\C)$  and $Z$ is locally isomorphic to $\C^d/G$ around $z$.
 \end{dfn}

Roughly speaking the local geometry of orbifolds reduces to local $G$-invariant geometry. 

We have a  complex of differential forms $(A^{\bullet}(Z),d)$ and a double complex  $(A^{\bullet,\bullet}(Z), \partial, \bar{\partial})$ of  bigraded differential forms which define the de Rham and the Dolbeault cohomology groups (for a fixed $p\in \N$) respectively:

$$H^{\bullet}_{dR}(Z,\C):=\frac{\ker d}{{\rm im}\,  d} \,\,\ \ {\rm and}\,\, \ \ H^{p,\bullet}(Z,\bar{\partial}):= \frac{\ker \bar{\partial}}{{\rm im}\,   \bar{\partial}}$$

\begin{thm}[Theorem 3.4.4 in \cite{CaramelloJr} and Theorem 1.2 in \cite{ANGELLA2013117} ] Let $Z$ be a compact complex orbifold. There are natural isomorphisms:

\begin{itemize}
    \item $H^{\bullet}_{dR}(Z,\C)\simeq H^{\bullet}(Z,\C)$
    \item $H^{p,\bullet}(Z,\bar{\partial})\simeq H^{\bullet}(X, \Omega_Z^p)  $
\end{itemize}

\end{thm}

\section{(1,1)-Lefschetz theorem for projective toric orbifolds} \label{1,1tor}

\begin{dfn} A subvariety $X\subset \Pj_{\Sigma}^d$ is quasi-smooth if $\mathbb{V}(I_X)\subset \mathbb{A}^{\#\Sigma(1)}$ is smooth outside $Z(\Sigma)$.
 \end{dfn}

\begin{ex} Quasi-smooth hypersurfaces  or more generally quasi-smooth intersection subvarieties are quasi-smooth subvarieties  (see \cite{CoxBatyrev} or  \cite{Mavlyutov} for more details).

\end{ex}

\begin{rmk}Quasi-smooth subvarieties are suborbifolds of $\Pj_{\Sigma}^d$ in the sense of Satake in \cite{Satake}. Intuitively speaking they are subvarieties whose only singularities come from the ambient space.

\end{rmk}

\begin{thm}\label{1,1}Let $X\subset \Pj_{\Sigma}^d$ be a quasi-smooth subvariety. Then every $(1,1)$-cohomology class $\lambda\in H^{1,1}(X)\cap H^2(X,\Z)$ is algebraic
\end{thm}
\begin{proof} From the exponential short exact sequence

$$0\rightarrow  \Z \rightarrow \cO_X\rightarrow \cO^*_{X}\rightarrow 0 $$

we have a long exact sequence in cohomology 

$$H^1(\cO^*_X)\rightarrow H^2(X,\Z)\rightarrow H^2(\cO_X)\simeq H^{0,2}(X)$$
where the last isomorphisms is due to Steenbrink in \cite{Steenbrink}. Now, it is enough to prove the commutativity of the next diagram 

$$\xymatrix{
H^2(X,\Z)\ar[r]\ar[d] & H^2(X,\cO_X)\ar[dd]_{\simeq}^{Dolbeault}\\
H^2(X,\C)\ar[d]_{de\, Rham}^{\simeq} &{}\\
H^2_{dR}(X,\C)\ar[r] & H^{0,2}_{\bar{\partial}}(X)
}
$$

The key points are the de Rham and  Dolbeault’s isomorphisms for orbifolds. The rest of the proof follows as  the $(1,1)$-Lefschetz theorem  in \cite{GriffithsHarris}. 

\end{proof}

\begin{rmk}\label{1,1} For $k=1$ and $\Pj_{\Sigma}^d$ as the projective space, we recover the classical $(1,1)$-Lefschetz theorem.

\end{rmk}

By the Hard Lefschetz Theorem for projective orbifolds (see \cite{Wang_2009} for details) we get an isomorphism of cohomologies :

$$H^{\bullet}(X,\Q)\simeq H^{2\dim X-\bullet}(X,\Q)$$
given by the Lefschetz morphism and since it is a morphism of Hodge structures, we have:
\begin{equation}\label{hard}
 H^{1,1}(X,\Q)\simeq H^{\dim X-1,\dim X-1}(X,\Q)   
\end{equation}

For $X$ as before:
\begin{cor}\label{dim} If the dimension of $X$ is $1$, $2$ or $3$. The Hodge conjecture holds on $X$. 

\end{cor}

\begin{proof} If the $dim_{\C} X=1$ the result is clear by the Hard Lefschetz theorem for projective orbifolds. The dimension 2 and 3 cases are covered by Theorem \ref{1,1} and the Hard Lefschetz. theorem. 
\end{proof}

\section{Cayley trick and Cayley proposition}

The Cayley trick is a way to associate to a quasi-smooth intersection subvariety a quasi-smooth hypersurface. Let $L_1,\dots , L_s$ be line bundles on $\Pj^d_{\Sigma}$ and let $\pi: \Pj(E)\rightarrow \Pj_{\Sigma}^d$ be the projective space bundle associated to the vector bundle $E=L_1\oplus \cdots \oplus L_s$. It is known that $\Pj(E)$ is a $(d+s-1)$-dimensional simplicial toric variety whose fan depends on the degrees of the line bundles and the fan $\Sigma$. Furthermore, if the Cox ring, without considering the grading, of $\Pj^d_{\Sigma}$ is $\C[x_1,\dots,x_m]$ then the Cox ring of $\Pj(E)$ is 
$$\C[x_1,\dots,x_m,y_1,\dots, y_s] $$

Moreover for $X$ a quasi-smooth intersection subvariety cut off by $f_1,\dots, f_s$ with $\deg(f_i)=[L_i]$ we relate the hypersurface $Y$ cut off by $F=y_1f_1+\dots +y_sf_s$  which turns out to be quasi-smooth. For more details see Section 2 in \cite{Mavlyutov}.  

We will denote  $\Pj(E)$ as $\Pj^{d+s-1}_{\Sigma, X}$ to keep track of its relation with $X$ and $\Pj^{d}_{\Sigma}$. 

The following is a key remark. 

\begin{rmk}\label{obs} There is a morphism $\iota: X\rightarrow Y\subset \Pj^{d+s-1}_{\Sigma, X}$. Moreover every point $z:=(x,y)\in Y$ with $y\neq 0$ has a preimage. Hence for any  subvariety $W=\mathbb{V}(I_W)\subset X\subset \Pj^d_{\Sigma}$ there exists $W'\subset Y\subset \Pj^{d+s-1}_{\Sigma,X}$ such that $\pi(W')=W$, i.e., $W'=\{z=(x,y)\mid x\in W \}$.

\end{rmk} 

For $X\subset \Pj_{\Sigma}^d$ a quasi-smooth intersection variety the morphism in cohomology induced by the inclusion $i^*: H^{d-s}(\Pj^d_{\Sigma},\C)\rightarrow H^{d-s}(X,\C)$ 
is injective by Proposition 1.4 in \cite{Mavlyutov}. 

\begin{dfn}\label{primitive} The primitive cohomology of $H^{d-s}_{\rm prim}(X)$ is the quotient $H^{d-s}(X,\C)/ i^*(H^{d-s}(\Pj^d_{\Sigma},\C))$ and $H^{d-s}_{\rm prim}(X,\Q)$ with rational coefficients.

\end{dfn}

$H^{d-s}(\Pj_{\Sigma}^d, \C)$ and $H^{d-s}(X, \C)$ have pure Hodge structures, and the morphism $i^*$ is compatible with them, so that  $H^{d-s}_{\rm prim}(X)$ gets a pure Hodge structure. 

The next Proposition is the Cayley proposition. 

\begin{prop}\label{Cayley}[Proposition 2.3 in \cite{BruzzoMontoya} ] Let $X=X_1 \cap \dots\cap X_s$ be a quasi-smooth intersection subvariety in $\Pj_{\Sigma}^d$ cut off by homogeneous polynomials $f_1\dots f_s$. Then for  $p\neq \frac{d+s-1}{2}, \frac{d+s-3}{2} $
$$H^{p-1,d+s-1-p}_{\mbox{\rm\footnotesize prim}}(Y)\simeq H^{p-s,d-p}_{\mbox{\rm\footnotesize prim}}(X). $$

\end{prop}

\begin{cor} If $d+s=2(k+1)$, 

$$H^{k+1-s,k+1-s}_{\rm prim}(X)\simeq H^{k,k}_{\rm prim}(Y)$$

\end{cor}

\begin{rmk} The above isomorphisms are also true with rational coefficients since $H^{\bullet}(X,\C)=H^{\bullet}(X,\Q)\otimes_{\Q} \C.$ See the beginning of Section 7.1 in \cite{Voisin_2002} for more details.

\end{rmk}

\section{Main result}

\begin{thm}
Let $Y=\{F=y_1f_1+\cdots+y_kf_k=0\}\subset \Pj^{2k+1}_{\Sigma,X}$ be the quasi-smooth hypersurface associated to the quasi-smooth intersection surface $X=X_{f_1}\cap \dots\cap X_{f_k}\subset \Pj^{k+2}_{\Sigma}$. Then on $Y$ the Hodge conjecture holds.
\end{thm}

\begin{proof}  If $H^{k,k}_{\rm prim}(X,\Q)=0$ we are done. So  let us assume  $H^{k,k}_{\rm prim}(X,\Q)\neq 0$. By the Cayley proposition $H^{k,k}_{\rm prim}(Y,\Q)\simeq H^{1,1}_{\rm prim}(X,\Q)$ and by the $(1,1)$-Lefschetz theorem for projective toric orbifolds there is a non-zero algebraic basis  $\lambda_{C_1}, \dots, \lambda_{C_n}$ with rational coefficients of  $H^{1,1}_{\rm prim}(X,\Q)$, that is, there are $n:=h^{1,1}_{\rm prim}(X,\Q)$ algebraic curves $C_1,\dots, C_n$ in $X$ such that under the Poincaré duality the class in homology $[C_i]$ goes to $\lambda_{C_i}$, $[C_i]\mapsto \lambda_{C_i}$. Recall that the Cox ring of $\Pj^{k+2}$ is contained in the Cox ring of $\Pj_{\Sigma,X}^{2k+1}$ without considering the grading. Considering the grading  we have that if $\alpha \in {\rm Cl}(\Pj_{\Sigma}^{k+2})$ then $(\alpha,0)\in {\rm Cl}(\Pj^{2k+1}_{\Sigma,X})$. So the polynomials defining $C_i\subset \Pj^{k+2}_{\Sigma}$ can be interpreted in $\Pj_{X,\Sigma}^{2k+1}$ but with different degree. Moreover,  by Remark \ref{obs} each $C_i$ is contained in $Y=\{F=y_1f_1+\cdots+y_kf_k=0\}$ and furthermore it has  codimension $k$. 

\textbf{Claim:} $\{\lambda_{C_i}\}_{i=1}^n$ is a basis of $H^{k,k}_{\rm prim}(Y,\Q)$.\\
It is enough to prove that $\lambda_{C_i}$ is different from zero in   $H^{k,k}_{\rm prim}(Y,\Q)$ or equivalently  that the cohomology classes $\{ \lambda_{C_i} \}^n_{i=1}$ do not come from the ambient space. By contradiction, let us assume that there exists a $j$ and $C\subset \Pj^{2k+1}_{\Sigma, X}$ such that $\lambda_{C}\in H^{k,k}(\Pj^{2k+1}_{\Sigma,X},\Q)$ with $i^{*}(\lambda_{C})=\lambda_{C_j}$  or in terms of homology there exists a $(k+2)$-dimensional algebraic subvariety  $V\subset \Pj^{2k+1}_{\Sigma, X}$ such that $V\cap Y=C_j$ so they are equal as a homology class of $\Pj^{2k+1}_{\Sigma, X}$,i.e., $[V\cap Y]=[C_j]$ . It is easy to check that $\pi(V)\cap X= C_j$ as a subvariety of $\Pj_{\Sigma}^{k+2}$ where $\pi: (x,y)\mapsto x $. Hence $[\pi(V)\cap X]= [C_j]$ which is equivalent to say that $\lambda_{C_j}$ comes from $\Pj^{k+2}_{\Sigma}$ which contradicts the choice of $[C_{j}]$.

\end{proof}

\begin{rmk} Into the proof of the previous theorem, the key fact was that on $X$  the Hodge conjecture holds and we translate it to $Y$ by contradiction. So, using an analogous argument we have:
    
\end{rmk}

\begin{prop}\label{meta}
Let $Y=\{F=y_1f_s+\cdots+y_sf_s=0\}\subset \Pj^{2k+1}_{\Sigma,X}$ be the quasi-smooth hypersurface associated to a quasi-smooth intersection subvariety $X=X_{f_1}\cap \dots\cap X_{f_s}\subset \Pj^{d}_{\Sigma}$ such that $d+s=2(k+1)$. If the Hodge conjecture holds on $X$  then it holds as well on $Y$. 
\end{prop}

\begin{cor}If the dimension of $Y$ is $2s-1$, $2s$ or $2s+1$ then the Hodge conjecture holds on $Y$.

\end{cor}
\begin{proof} By Proposition \ref{meta} and Corollary \ref{dim}.

\end{proof}

\bibliographystyle{acm}
\bibliography{references}

\end{document}